\documentclass[letterpaper, 10 pt, conference]{ieeeconf}  

\IEEEoverridecommandlockouts                              

\usepackage{graphicx} 
\usepackage{amsmath} 
\usepackage{amssymb,cite}  

\usepackage{amsfonts,nccmath,mathtools,mathrsfs,color,balance}

\usepackage[colorlinks = true,
            linkcolor = blue,
            urlcolor  = blue,
            citecolor = blue,
            anchorcolor = blue]{hyperref}

\newcommand{\norm}[1]{\left\|#1\right\|}

\newtheorem{remark}{Remark}
\newtheorem{corollary}{Corollary}
\newtheorem{lemma}{Lemma}
\newtheorem{example}{Example}

\newtheorem{assumption}{Assumption}

\newtheorem{theorem}{Theorem}
\newcommand{\bmat}[1]{\begin{bmatrix}#1\end{bmatrix}}
\usepackage[nameinlink]{cleveref}
\crefname{equation}{}{}
\crefname{theorem}{Theorem}{Theorems}
\crefname{corollary}{Corollary}{Corollaries}
\crefname{example}{Example}{Examples}
\crefname{assumption}{Assumption}{Assumptions}
\crefname{lemma}{Lemma}{Lemmas}
\crefname{proposition}{Proposition}{Propositions}
\crefname{figure}{Figure}{Figures}
\crefname{table}{Table}{Tables}
\crefname{section}{Section}{Sections}
\crefname{appendix}{Appendix}{Appendices}

\newtheorem{proposition}{Proposition}
\newcommand{\tr}{{{\mathsf T}}}
\newcommand{\mK}{{\mathsf{K}}}
\newcommand{\mM}{{\mathsf{M}}}
\newcommand{\mZ}{{\mathsf{Z}}}

\DeclareMathOperator*{\Tr}{Tr}

\title{\LARGE \bf
Connectivity of the Feasible and Sublevel Sets of Dynamic Output Feedback Control with Robustness Constraints
}

\author{Bin Hu$^{1}$, and Yang Zheng$^{2}$
\thanks{B. Hu is generously supported by the NSF award CAREER-2048168 and the 2020 Amazon
research award.}
\thanks{$^{1}$Bin Hu is with the Coordinated Science Laboratory (CSL) and the Department of Electrical and Computer Engineering, University of Illinois at Urbana-Champaign,
        {\tt\small 
binhu7@illinois.edu
}}%
 \thanks{$^{3}$Yang Zheng is  with  the Department of Electrical and Computer Engineering, University of California San Diego,
        {\tt\small zhengy@eng.ucsd.edu}}%
}

\begin{document}

\maketitle
\thispagestyle{empty}
\pagestyle{empty}

\begin{abstract}%
This paper considers the optimization landscape of linear dynamic output feedback control with $\mathcal{H}_\infty$  robustness constraints.  
 We consider the feasible set of all the stabilizing full-order dynamical controllers that satisfy an additional $\mathcal{H}_\infty$ robustness constraint.  
 We show that  this $\mathcal{H}_\infty$-constrained set  has at most two path-connected components that are diffeomorphic under a mapping defined by a similarity transformation. Our proof technique utilizes a classical change of variables in $\mathcal{H}_\infty$ control to establish a subjective mapping from a set with a convex projection to the  $\mathcal{H}_\infty$-constrained set. This proof idea can also be used to establish the same topological properties of strict sublevel sets of linear quadratic Gaussian (LQG) control and optimal $\mathcal{H}_\infty$ control. 
 Our results bring positive news for gradient-based policy search on robust control problems.
\end{abstract}


\section{Introduction}
\label{sec:intro}
Inspired by the impressive successes of reinforcement learning, model-free policy optimization techniques are receiving renewed interests from the controls field. 
Indeed, we have seen significant recent advances on understanding the theoretical properties of policy optimization methods~on benchmark control problems, 
such as linear quadratic regulator (LQR) \cite{pmlr-v80-fazel18a,malik2019derivative,mohammadi2021convergence,furieri2020learning}, 
linear robust control \cite{zhang2021policy,zhang2020stability,gravell2020learning,zhang2021derivative}, 
and Markov jump linear quadratic control~\cite{jansch2020convergence,rathod2021global,jansch2020policyA}.

It is well-known that all these control problems are~non-convex in the policy space. Classical control theory typically parameterizes the control policies into a convex domain over which efficient optimization algorithms exist~\cite{zhou1996robust}.~An important recent discovery is that despite non-convexity, many state-feedback control problems (e.g., LQR) admit a useful property of \textit{gradient dominance}~\cite{pmlr-v80-fazel18a}. Therefore, model-free policy search methods are guaranteed to enjoy global convergence for these problems~\cite{pmlr-v80-fazel18a,furieri2020learning,jansch2020convergence}. Note that most convergence results require a direct access of the underlying system state, in which a simple change of variables exist to get a convex reformulation of the control problems~\cite{sun2021learning}.     


For real-world control applications, however, we may only have access to partial output measurements. 
In the output feedback case, the theoretical results for direct policy search are much fewer and far less complete~\cite{feng2019exponential,fatkhullin2020optimizing,zheng2021analysis,duan2022optimization,umenberger2022globally}. It remains unclear whether model-free policy gradient methods can be modified to yield global convergence guarantees. It has been revealed that the set of stabilizing static output-feedback controllers can be highly disconnected \cite{feng2019exponential}. This is quite different from the state feedback case \cite{bu2019topological}. Such a negative result indicates that the performance of gradient-based policy search on static output feedback control highly depends on the initialization,  and
only convergence to stationary points  has been established \cite{fatkhullin2020optimizing}.
It is thus natural to investigate dynamical controllers for the output feedback case, and to see whether the corresponding optimization landscape is more favorable for direct policy search methods. 
%
The very recent work~\cite{zheng2021analysis} shows that  the set of stabilizing full-order dynamical controllers~has at
most two path-connected components that are identical in the frequency domain. 
This brings some positive news and opens the possibility of developing global convergent policy search methods for dynamical output feedback problems, such as linear quadratic Gaussian (LQG) control ~\cite{zheng2021analysis}. Two  other recent studies are~\cite{duan2022optimization,umenberger2022globally}.  In \cite{umenberger2022globally}, the global convergence of policy search over dynamical filters was proved for a simpler estimation~problem.  

It is well-known that the optimal LQG controller has no robustness guarantee~\cite{doyle1978guaranteed}. It is thus important to explicitly incorporate robustness constraints for the search of dynamical controllers. In this paper, we study the topological properties of the feasible set for linear dynamical output feedback control with $\mathcal{H}_\infty$ robustness constraints.
The $\mathcal{H}_\infty$ constraints have been widely used in robust
control \cite{zhou1996robust,Dullerud99} and risk-sensitive control \cite{whittle1990risk}.
Our main result shows that the set of all stabilizing full-order dynamical  controllers satisfying an additional input-output $\mathcal{H}_\infty$ constraint has at
most two path-connected components, and they are diffeomorphic under a mapping defined by a
similarity transformation. Our proof technique is inspired by~\cite{zheng2021analysis} and relies on a non-trivial but known change of variables for $\mathcal{H}_\infty$ control~\cite{gahinet1994linear,scherer1997multiobjective}. If the control cost is invariant under similarity transformation, one can initialize the local policy search anywhere within the feasible set and there is always a continuous path connecting the initial point to a global minimum.
Our result sheds new light on model-free policy search for robust control tasks.

The rest of this paper is organized as follows. In \Cref{section:problem_statement}, we formulate the linear dynamic output feedback control with $\mathcal{H}_\infty$ constraints as a constrained policy optimization problem. 
\Cref{sec:connectivity} presents our main theoretical results. We revisit connectivity of strict sublevel sets for LQG and $\mathcal{H}_\infty$ control in \Cref{section:sublevel-set}. Some illustrative examples are shown in \Cref{section:examples}. We conclude the paper in  \cref{section:conclusion}. Some auxillary proofs and results are provided in the appendix.

{\textit{Notations}}: 
The set of  $k\times k$ real symmetric matrices is denoted by $\mathbb{S}^k$, and the determinant of a square matrix $M$ is denoted by $\det M$.  We use $I_k$ to denote the $k\times k$ identity matrix, and use $0_{k_1\times k_2}$ to denote the $k_1\times k_2$ zero matrix; we sometimes omit their dimensions if they are clear from the context.   Given a matrix $M \in \mathbb{R}^{k_1 \times k_2}$, $M^\tr$ denotes the transpose of $M$. For any $M_1,M_2\in\mathbb{S}^k$, we use $M_1\prec M_2$ ($M_1\preceq M_2$) and $M_2\succ M_1$ ($M_2\succeq M_1$) to mean that $M_2-M_1$ is positive (semi)definite. 





\section{Preliminaries and Problem Statement} \label{section:problem_statement}

%

\subsection{Dynamic output feedback with $\mathcal{H}_\infty$ constraints}
We consider a continuous-time linear dynamical system\footnote{All topological results can be extended to the discrete-time domain.} 
\begin{equation}\label{eq:Dynamic}
\begin{aligned}
\dot{x}(t) &= Ax(t)+B_1 w(t)+ B_2 u(t), \\
z(t)&= C_1 x(t)+D_{11} w(t)+ D_{12} u(t),\\
y(t) &= C_2 x(t)+D_{21}w(t),
\end{aligned}
\end{equation}
where $x(t) \in \mathbb{R}^{n_x}$ is the state, $u(t)\in \mathbb{R}^{n_u}$ is the control action, $w(t)\in \mathbb{R}^{n_w}$ is the exogenous disturbance, $y(t) \in \mathbb{R}^{n_y}$ is the measured output, and $z(t)\in \mathbb{R}^{n_w}$ is the regulated performance output. 
We make the following assumption.
\begin{assumption} \label{assumption:stabilizability}
The state-space model $(A, B_2, C_2)$ in \cref{eq:Dynamic} is stabilizable and detectable.
\end{assumption}

We aim to design a controller that maps the measured output to the control action, in order to minimize some control performance metric, while satisfying stability and/or robustness constraints.
Such control design problems can be formulated as a constrained policy optimization of the form 
\begin{equation} \label{eq:policy-optimization}
\min_{\mK\in\mathcal{K}} \;\; J(\mK),
\end{equation}
where the decision variable $\mK$ is determined by the policy parameterization,  the objective function $J(\mK)$ measures the closed-loop performance, and the feasible set $\mathcal{K}$ is specified by some stability/robustness requirements. 
We consider the following policy parameterization and robustness constraint:

\begin{itemize}
\item \textit{Decision variable }$\mK$: Output feedback control problems typically require dynamical controllers, and we consider the full-order dynamical controller in the form of:
\begin{align}\label{eq:Dynamic_Controller}
\begin{split}
        \dot \xi(t) &= A_{\mK}\xi(t) + B_{\mK}y(t), \\
        u(t) &= C_{\mK}\xi(t)+ D_{\mK} y(t),
\end{split}
\end{align}
where $\xi(t)$ is the controller state with the same dimension as $x(t)$, and matrices $(A_{\mK}, B_{\mK}, C_{\mK}, D_{\mK})$ specify the controller dynamics.
For convenience, we denote 
\begin{align}
\mK:=\begin{bmatrix}
    D_{\mK} & C_{\mK} \\
    B_{\mK} & A_{\mK}
    \end{bmatrix} \in \mathbb{R}^{(n_u+n_x)\times (n_y+n_x)},
\end{align}
but this matrix $\mK$ should be interpreted as the dynamical controller in \cref{eq:Dynamic_Controller}. 
\item \textit{Feasible region}:  The controller $\mK$ needs to stabilize the closed-loop system and satisfy a robustness constraint that enforces the $\mathcal{H}_\infty$ norm of the transfer function from $w(t)$ to $z(t)$ smaller than a pre-specified level $\gamma$. 
\end{itemize}

We allow a general cost function $J(\mK)$, which can be an $\mathcal{H}_2$ performance on some other performance channel, or more general user-specified performance metrics.  
One advantage for the policy optimization formulation \cref{eq:policy-optimization} is that it opens the possibility of solving robust control design via model-free policy search methods. This paper aims to characterize connectivity of $\mathcal{K}$ and strict sublevel sets of~$J(\mK)$. 

\subsection{Problem statement}


We denote the state of the closed-loop system as $\zeta=\bmat{x^\tr & \xi^\tr}^\tr$ after combining~\cref{eq:Dynamic_Controller} with~\cref{eq:Dynamic}. It is not difficult to derive the closed-loop system
\begin{align}\label{eq:closed-loop_LQG}
\begin{split}
\dot{\zeta}(t)&=A_{\mathrm{cl}} \zeta(t)+ B_{\mathrm{cl}} w(t),\\
w(t)&=C_{\mathrm{cl}} \zeta(t)+ D_{\mathrm{cl}} w(t),
\end{split}
\end{align}
where the matrices $(A_{\mathrm{cl}},B_{\mathrm{cl}},C_{\mathrm{cl}},D_{\mathrm{cl}})$ are given by
\begin{align}\label{eq:closeABCD}
\begin{split}
A_{\mathrm{cl}}&:=\begin{bmatrix}
    A+B_2 D_{\mK} C_2  & B_2 C_{\mK} \\
    B_{\mK}C_2 & A_{\mK}
\end{bmatrix},\\
B_{\mathrm{cl}}&:=\begin{bmatrix} B_1+B_2 D_{\mK} D_{21} \\ B_{\mK} D_{21} \end{bmatrix},\\
C_{\mathrm{cl}}&:=\begin{bmatrix} C_1+D_{12}D_{\mK} C_2  &  D_{12} C_{\mK} \end{bmatrix},\\
 D_{\mathrm{cl}}&:=D_{11}+D_{12}D_{\mK} D_{21}.
\end{split}
\end{align}
The closed-loop system  is internally stable  if and only if $A_{\mathrm{cl}}$ is~Hurwitz~\cite{zhou1996robust}.
The set of full-order stabilizing dynamical controllers is thus defined as
\begin{equation} \label{eq:internallystabilizing}
    \mathcal{C}_{\mathrm{stab}} := \left\{
    \left.\mK
    \in \mathbb{R}^{(n_u+n_x) \times (n_y+n_x)} \right|\;  A_{\mathrm{cl}}~\text{is Hurwitz}\right\}.
\end{equation}

The transfer function from $w(t)$ to $z(t)$ is 
\begin{equation} \label{eq:hinf-transfer-function}
    \mathbf{T}_{zw}(s) = C_{\mathrm{cl}}(sI-A_{\mathrm{cl}})^{-1} B_{\mathrm{cl}}+D_{\mathrm{cl}}. 
\end{equation}
Then, the feasible set  is formally specified as
\begin{equation} \label{eq:Hinfset}
    \mathcal{K}_{\gamma} := \left\{
    \left.\mK
    \in \mathcal{C}_{\mathrm{stab}}\right|\;   \norm{\mathbf{T}_{zw}}_\infty< \gamma\right\},
\end{equation}
%
where $\|\mathbf{T}_{zw}\|_{\infty}$ denotes the $\mathcal{H}_{\infty}$ norm of $\mathbf{T}_{zw}$, and can be calculated as  $\|\mathbf{T}_{zw}\|_{\infty}:= \sup_{\omega} {\sigma}_{\max} (\mathbf{T}_{zw}(j\omega))$, with ${\sigma}_{\max}(\cdot)$ denoting the maximum singular value. 
In \cref{eq:Hinfset}, we explicitly highlight the robustness level $\gamma$ via the subscript.
Under~\cref{assumption:stabilizability}, there exists a finite positive value 
$$
\gamma^\star : = \inf_{\mK} \;\; \norm{\mathbf{T}_{zw}}_\infty.
$$ 
Then, $\mathcal{K}_\gamma$ is non-empty if and only if $\gamma>\gamma^\star$.
Obviously, we have $\mathcal{K}_{\gamma_0}\subset \lim_{\gamma\rightarrow \infty} \mathcal{K}_\gamma=\mathcal{C}_{\mathrm{stab}}$ for any positive $\gamma_0$. 

In \cref{eq:policy-optimization}, it is possible to estimate the gradient of $J(\mK)$ and 
$\norm{\mathbf{T}_{zw}}_\infty$
from sampled system trajectories, and one may apply model-free gradient-based barrier algorithms to find a solution in an iterative fashion. To understand the performance of such model-free policy search algorithms, we need to characterize the optimization landscape of \cref{eq:policy-optimization}. In particular, we focus on some geometrical properties of the feasible region $\mathcal{K}_{\gamma}$ and strict sublevel sets of $J(\mK)$.  
It~is well-known that $\mathcal{K}_{\gamma}$ is in general non-convex, but little is known about their other geometrical properties. 
Even for the case $\gamma \to \infty$, only a very recent work shows that 
%
$\mathcal{C}_{\mathrm{stab}}$ has at
most two path-connected components that are identical up to similarity transformations~\cite[Theorems 3.1 \& 3.2]{zheng2021analysis}. 

In many cases, it is desirable to explicitly encode some robustness guarantee for the feasible region~\cite{whittle1990risk,Dullerud99,doyle1978guaranteed}. However, the connectivity of the $\mathcal{H}_\infty$-constrained set $\mathcal{K}_{\gamma}$ remains unknown. In this paper, we focus on topological properties of $\mathcal{K}_\gamma$ and their implications to gradient-based policy search.
We will show that $\mathcal{K}_\gamma$ shares similar properties with $\mathcal{C}_{\mathrm{stab}}$.

\begin{remark}
The dynamical controller \cref{eq:Dynamic_Controller} is proper. Depending on the cost function $J(\mK)$ (e.g., LQG~\cite{zhou1996robust}), we may want to confine the policy space to strictly proper dynamical controllers. Then the feasible set is defined as
\begin{equation} \label{eq:Hinfset1}
    \tilde{\mathcal{K}}_{\gamma} := \left\{
    \left.\mK
    \in \mathcal{K}_\gamma\right|\;   D_{\mK}=0\right\}.
\end{equation}
Our analysis technique works for both $\tilde{\mathcal{K}}_\gamma$ and $\mathcal{K}_\gamma$, and we show that $\tilde{\mathcal{K}}_\gamma$ and $\mathcal{K}_\gamma$ have similar topological properties. \hfill $\square$
\end{remark}

\section{Path-connectivity of ${\mathcal{K}}_{\gamma}$}
\label{sec:connectivity}

In this section, we present our main results on the topological properties of $\mathcal{K}_{\gamma}$. 
We first have a simple observation.
\begin{lemma} \label{lemma:unbounded_proper}
    Let $\gamma>\gamma^\star$. The set ${\mathcal{K}}_{\gamma}$ is non-empty, open,  unbounded and non-convex.
\end{lemma}

This fact is well-known. Then openness of ${\mathcal{K}}_{\gamma}$ follows from the continuity of the $\mathcal{H}_\infty$ norm. It is unbounded since $\mathcal{H}_\infty$ norm is invariant under similarity transformations that are unbounded in the state-space domain. The non-convexity is also known, and we illustrate it using the example below. 

\begin{example}\label{example1}
Consider an open-loop unstable dynamical system \cref{eq:Dynamic} with
$A=B_1=B_2=C_1=C_2=D_{21}=D_{12}=1$, and  $D_{11}=0
$.
It is easy to verify that the following dynamical controllers
$$
\mK^{(1)}=\bmat{0 & 2\\ -2 & -2}, \,\mK^{(2)}=\bmat{0 & -2\\ 2 & -2}
$$
satisfy $\norm{C_{\mathrm{cl}}(sI-A_{\mathrm{cl}})^{-1} B_{\mathrm{cl}}+D_{\mathrm{cl}}}_\infty < 3.33$, and thus 
we have $\mK^{(1)} \in \mathcal{K}_{3.33}, \mK^{(2)} \in \mathcal{K}_{3.33}$. However, 
$$
\frac{1}{2}(\mK^{(1)}+\mK^{(2)}) = \bmat{0 & 0 \\ 0  & -2},
$$
fails to stabilize the system, and thus is outside $\mathcal{K}_{3.33}$.
\hfill $\square$
\end{example}

Despite the non-convexity, $\mathcal{K}_\gamma$ has some nice connectivity property which will be established next.


\subsection{Main results}
Our first main technical result is stated as follows.

\begin{theorem} \label{Theo:disconnectivity}
Given any $\gamma > \gamma^\star$, the set ${\mathcal{K}}_{\gamma}$ has at most two path-connected components.  
\end{theorem}

Before presenting a formal proof for \Cref{Theo:disconnectivity},  we first give some high-level ideas. Based on the bounded real lemma \cite{Dullerud99}, 
we have $\mK\in \mathcal{K}_\gamma$ if and only if the matrix inequality,
\begin{align}
\label{eq:LMI1}
 \bmat{A_{\mathrm{cl}}^\tr P+P A_{\mathrm{cl}} & P B_{\mathrm{cl}} & C_{\mathrm{cl}}^\tr\\ B_{\mathrm{cl}}^\tr P & -\gamma I & D_{\mathrm{cl}}^\tr\\ C_{\mathrm{cl}} & D_{\mathrm{cl}} & -\gamma I}    \prec 0, \,\, P\succ 0,
\end{align}
is feasible. 
Clearly, the condition \cref{eq:LMI1} is not convex in $\mK$ and $P$.
Our result in \Cref{Theo:disconnectivity} relies on the fact that \cref{eq:LMI1} can be convexified into a linear matrix inequality (LMI) (that is convex and hence path-connected), using a non-trivial but known change of variables for $\mathcal{H}_\infty$ control~\cite{gahinet1994linear,scherer1997multiobjective}. The only potential of disconnectivity comes from the fact that the set of invertible matrices corresponding to similarity transformations has two path-connected components. Our proof is inspired by the recent work~\cite{zheng2021analysis} that characterizes $\mathcal{C}_{\mathrm{stab}}$ only, with the main difference being that we need to analyze a more complicated  $\mathcal{H}_\infty$ constraint \cref{eq:LMI1}.   

We now illustrate this idea for the case of  state feedback (i.e. $y(t)=x(t)$ and $u(t)=Kx(t)$ with $K \in \mathbb{R}^{n_u \times n_x}$). In this case,  it is known that  \cref{eq:LMI1} is feasible\footnote{In the state-feedback case,  $(A_{\mathrm{cl}}, B_{\mathrm{cl}}, C_{\mathrm{cl}}, D_{\mathrm{cl}})$ should be calculated from some formulas which are different from \cref{eq:closeABCD}. We omit the details.} if and only if 
\begin{align}\label{eq:lmi2}
M_\gamma(Q,L)\prec 0, \,\, Q\succ 0
\end{align}
 is feasible, where $M_\gamma(Q,L)$ is defined as 
\begin{align*}
\bmat{QA^\tr+AQ+L^\tr B_2^\tr+B_2 L  & B_1 & (C_1 Q+D_{12} L)^\tr\\ B_1^\tr & -\gamma I & 0\\C_1 Q+ D_{12} L & 0 & -\gamma I}.
\end{align*}
Using a simple change of variables $K =LQ^{-1}$, we have
     $$
     \begin{aligned}
     &\{K\in \mathbb{R}^{n_x \times n_u} \mid  \text{\cref{eq:LMI1} is feasible}\} \\
        \; \Longleftrightarrow \;\; &\{K = LQ^{-1} \in \mathbb{R}^{n_x \times n_u} \mid \text{\cref{eq:lmi2} is satisfied}\}.
     \end{aligned}
        $$
Since the set of $(Q,L)$ satisfying LMI \cref{eq:lmi2} is convex and the map 
     $   K = LQ^{-1}$  
    is continuous, the set $\{K\in \mathbb{R}^{n_x \times n_u} \mid \text{\cref{eq:LMI1} is feasible}\}$ is path-connected.

   The  analysis above hinges upon the fact that in the~state-feedback case, the non-convex condition \cref{eq:LMI1} can be convexified  using  the simple change of variables
     $K = LQ^{-1}$.    
In the output feedback case,  a similar condition can be derived using a more complicated change of variables in~\cite{scherer1997multiobjective}.  We will leverage this fact to prove 
 \cref{Theo:disconnectivity}. 
Specifically, it is known that
a controller $\mK\in \mathcal{K}_\gamma$ can be constructed from the solution of the following LMI condition:
\begin{align}\label{eq:lmi3}
\begin{bmatrix}
X & I \\ I & Y
\end{bmatrix} \! \succ 0,\ 
\mM_\gamma(X,Y,\hat{\mathbf{A}},\hat{\mathbf{B}},\hat{\mathbf{C}},\hat{\mathbf{D}})
\!\! \prec 0,
\end{align}
where $ X\in\mathbb{S}^{n_x}, Y\in\mathbb{S}^{n_x},
\hat{\mathbf{A}}\in\mathbb{R}^{n_x \times n_x}, \hat{\mathbf{B}}\in \mathbb{R}^{n_x \times n_y}, \hat{\mathbf{C}} \in \mathbb{R}^{n_u \times n_x}$, and $\hat{\mathbf{D}}\in \mathbb{R}^{n_u\times n_y},$ are decision variables. The linear mapping $\mM_\gamma(X,Y,\hat{\mathbf{A}},\hat{\mathbf{B}},\hat{\mathbf{C}},\hat{\mathbf{D}})$ is defined as
    \begingroup
    \setlength\arraycolsep{1.5pt}
\def\arraystretch{1.2}
\begin{align}
\mM_\gamma(X,Y,\hat{\mathbf{A}},\hat{\mathbf{B}},\hat{\mathbf{C}},\hat{\mathbf{D}})\!=\!\bmat{\mM_{11} & \mM_{12} & \mM_{13} & \mM_{14}\\ \mM_{12}^\tr & \mM_{22} & \mM_{23} & \mM_{24}\\ \mM_{13}^\tr & \mM_{23}^\tr & \mM_{33} & \mM_{34}\\ \mM_{14}^\tr & \mM_{24}^\tr & \mM_{34}^\tr & \mM_{44}},
\end{align}
\endgroup
where the blocks $\mM_{ij}$ are given by
\begin{align}
\begin{split}
\mM_{11}&=AX+XA^\tr+B_2 \hat{\mathbf{C}}+(B_2 \hat{\mathbf{C}})^\tr,\\
\mM_{12}&=\hat{\mathbf{A}}^\tr+(A+B_2 \hat{\mathbf{D}} C_2), \\
\mM_{13}&= B_1+B_2 \hat{\mathbf{D}} D_{21},\\
\mM_{14}&=(C_1 X+D_{12} \hat{\mathbf{C}})^\tr,\\
\mM_{22}&=A^\tr Y + YA+\hat{\mathbf{B}} C_2+(\hat{\mathbf{B}} C_2)^\tr, \\
\mM_{23}&=YB_1+\hat{\mathbf{B}} D_{21},\\
\mM_{24}&=(C_1+D_{12} \hat{\mathbf{D}} C_2)^\tr,\\
\mM_{33}&=-\gamma I,\\
\mM_{34}&=(D_{11}+D_{12} \hat{\mathbf{D}} D_{21})^\tr, \\
\mM_{44}&=-\gamma I.
\end{split}
\end{align}

%
Based on LMI \cref{eq:lmi3}, we  introduce two useful sets:
\begin{align} \label{eq:SetF}
\mathcal{F}_\gamma:=
\bigg\{&
(X,Y,\hat{\mathbf{A}},\hat{\mathbf{B}},\hat{\mathbf{C}},\hat{\mathbf{D}})
\mid  \text{\cref{eq:lmi3} is satisfied}
\bigg\}, \\
\begin{split}
\mathcal{G}_\gamma:=
\bigg\{&(X,Y,\hat{\mathbf{A}},\hat{\mathbf{B}},\hat{\mathbf{C}},\hat{\mathbf{D}},\Pi,\Xi)
\mid \Pi,\Xi\in\mathbb{R}^{n_x \times n_x},\, \\&(X,Y,\hat{\mathbf{A}},\hat{\mathbf{B}},\hat{\mathbf{C}},\hat{\mathbf{D}})\in\mathcal{F}_{\gamma},\, \Xi\Pi =I-YX\bigg\}.
\end{split}
\end{align}
It is obvious that $\mathcal{F}_\gamma$ is convex and hence path-connected. Together with the fact that the set of $n_x \times n_x$ invertible matrices has two path-connected components, this guarantees that $\mathcal{G}_\gamma$ has exactly two path-connected components.
We shall see that there exists a continuous surjective map from $\mathcal{G}_{\gamma}$ to $\mathcal{K}_\gamma$, and thus $\mathcal{K}_\gamma$ has at most two path-connected components. A detailed proof is provided in the next subsection.

\subsection{Detailed proof of \Cref{Theo:disconnectivity}} \label{subsection:proof-theorem-1}


\begin{lemma}\label{lemma:connectivity_preliminary}
For any $(X,Y,\hat{\mathbf{A}},\hat{\mathbf{B}},\hat{\mathbf{C}},\hat{\mathbf{D}},\Pi,\Xi)\in\mathcal{G}_{\gamma}$, $\Pi$ and $\Xi$ are always invertible, and consequently, the block triangular matrices
$
\begin{bmatrix} I & 0 \\ YB_2 & \Xi
\end{bmatrix}$ and $\begin{bmatrix} I & C_2 X\\
0 & \Pi  \end{bmatrix}
$
are invertible.
\end{lemma}

\vspace{1mm}

The proof is straightforward by observing that $\det(\Pi\Xi) = \det(YX-I) \neq 0$ for any $(X,Y,\hat{\mathbf{A}},\hat{\mathbf{B}},\hat{\mathbf{C}},\hat{\mathbf{D}},\Pi,\Xi)\in\mathcal{G}_{\gamma}$.
%
%
%
Based on the change of variables in \cite{scherer1997multiobjective}, we can map each element of $\mathcal{G}_{\gamma}$ back to a controller $\mK\in\mathbb{R}^{(n_u+n_x)\times(n_y+n_x)}$. 
    For each $\mZ=(X,Y,\hat{\mathbf{A}},\hat{\mathbf{B}},\hat{\mathbf{C}},\hat{\mathbf{D}},\Pi,\Xi)$ in $\mathcal{G}_{\gamma}$, we define
        \begingroup
    \setlength\arraycolsep{1.5pt}
\def\arraystretch{1.2}
    \begin{align} \label{eq:change_of_variables_output}
\begin{split}
        \Phi(\mZ)&= 
        \begin{bmatrix}
        \Phi_D(\mZ) & \Phi_C(\mZ) \\
        \Phi_B(\mZ) & \Phi_A(\mZ)
        \end{bmatrix}\\
        &=
        \begin{bmatrix} I & 0 \\ YB_2 & \Xi
        \end{bmatrix}^{-1} \begin{bmatrix}
        \hat{\mathbf{D}} & \hat{\mathbf{C}} \\
        \hat{\mathbf{B}} & \hat{\mathbf{A}}-YAX
        \end{bmatrix}\begin{bmatrix} I & C_2 X\\
        0 & \Pi  \end{bmatrix}^{-1}.
\end{split}
    \end{align}
\endgroup

We now present the following result which is essential for the proof of  \cref{Theo:disconnectivity}. 
\begin{proposition}\label{proposition:Phi_surjective}
The mapping $\Phi$ in~\cref{eq:change_of_variables_output} is a continuous and surjective mapping from $\mathcal{G}_{\gamma}$ to $\mathcal{K}_\gamma$.
\end{proposition}
\begin{proof}
It is clear that $\Phi(\cdot)$ is a continuous mapping.  
To show that $\Phi$ is a mapping onto $\mathcal{K}_\gamma$, we need to prove the following statements: 
\begin{enumerate}
    \item For any arbitrary controller $ \mK \in \mathcal{K}_\gamma$, there exists $\mZ=(X,Y,\hat{\mathbf{A}},\hat{\mathbf{B}},\hat{\mathbf{C}},\hat{\mathbf{D}},\Pi,\Xi)\in\mathcal{G}_{\gamma}$ such that 
    $
        \Phi(\mZ) = \mK
    $.
    \item For all $ \mZ=(X,Y,\hat{\mathbf{A}},\hat{\mathbf{B}},\hat{\mathbf{C}},\hat{\mathbf{D}},\Pi,\Xi)\in\mathcal{G}_{\gamma}$, we have $\Phi(\mZ)\in \mathcal{K}_\gamma$.
\end{enumerate}

To show the first statement, let $\mK= \begin{bmatrix}
D_{\mK} & C{_\mK} \\ B{_\mK} & A{_\mK}
\end{bmatrix}\in\mathcal{K}_\gamma$ be arbitrary. By the bounded real lemma \cite{Dullerud99}, there exists $P\succ 0$ such that \cref{eq:LMI1} is feasible. We partition the matrix~$P$~as
\begin{equation} \label{eq:Ydef}
P=\begin{bmatrix} Y & \Xi \\ \Xi^\tr & \hat{Y} \end{bmatrix}.
\end{equation}
 Without loss of generality, we assume that $\det \Xi\neq 0$ (otherwise we can add a small perturbation on $\Xi$ thanks to the strict inequality in \cref{eq:LMI1}).  We further define
\begin{align}\label{eq:Xdef}
\begin{bmatrix}
X & \Pi^\tr \\ \Pi & \hat{X}
\end{bmatrix}:=\begin{bmatrix} Y & \Xi \\ \Xi^\tr & \hat{Y} \end{bmatrix}^{-1},
\qquad
T := 
\begin{bmatrix} X & I \\ \Pi & 0 \end{bmatrix}.
\end{align}
we can verify that
\begin{equation}\label{eq:Gset_cond1}
YX+\Xi\Pi=I,
\qquad
T^\tr  P T = \begin{bmatrix} X & I \\ I & Y \end{bmatrix}
\succ 0.
\end{equation}
Now we choose $(\hat{\mathbf{A}},\hat{\mathbf{B}},\hat{\mathbf{C}},\hat{\mathbf{D}})$ as
\begin{align}\label{eq:change_of_var1} 
\begin{split}
\hat{\mathbf{A}} =& Y(A + B_2 D{_\mK}C_2)X + \Xi B{_\mK} C_2 X \\
& \qquad \qquad \qquad + YB_2 C{_\mK}\Pi+ \Xi A{_\mK}\Pi, \\
\hat{\mathbf{B}} =& YB_2 D{_\mK} + \Xi B{_\mK}, \\
\hat{\mathbf{C}}  =& D{_\mK}C_2 X + C{_\mK}\Pi,\\
\hat{\mathbf{D}} =& D{_\mK}.
\end{split}
\end{align}

We can then verify that $\mM_\gamma(X,Y,\hat{\mathbf{A}},\hat{\mathbf{B}},\hat{\mathbf{C}},\hat{\mathbf{D}})$ is exactly the same as
\begin{align*}
\bmat{T^\tr & 0 & 0\\ 0 & I & 0\\ 0 & 0 & I}  \bmat{A_{\mathrm{cl}}^\tr P+P A_{\mathrm{cl}} & P B_{\mathrm{cl}} & C_{\mathrm{cl}}^\tr\\ B_{\mathrm{cl}}^\tr P & -\gamma I & D_{\mathrm{cl}}^\tr\\ C_{\mathrm{cl}} & D_{\mathrm{cl}} & -\gamma I}  \bmat{T & 0 & 0\\ 0 & I & 0\\ 0 & 0 & I},
\end{align*}
which is clearly negative definite due to \cref{eq:LMI1}.
Thus, we have  $\mZ=(X,Y,\hat{\mathbf{A}},\hat{\mathbf{B}},\hat{\mathbf{C}},\hat{\mathbf{D}},\Pi,\Xi)\in\mathcal{G}_{\gamma}$ by the definition of $\mathcal{G}_{\gamma}$.
Note that  \cref{eq:change_of_var1} can be compactly rewritten as
$$
\begin{bmatrix} \hat{\mathbf{D}} & \hat{\mathbf{C}}\\ \hat{\mathbf{B}} & \hat{\mathbf{A}}-YAX \end{bmatrix}
=
\begin{bmatrix} I & 0 \\ YB_2 & \Xi \end{bmatrix} 
\begin{bmatrix} D{_\mK} & C{_\mK} \\ B{_\mK} & A{_\mK}\end{bmatrix}
\begin{bmatrix} I & C_2 X \\ 0 & \Pi \end{bmatrix}.
$$
Based on \cref{lemma:connectivity_preliminary}, we have
$$
\begin{bmatrix}
D{_\mK} & C{_\mK} \\ 
B{_\mK} & A{_\mK}
\end{bmatrix}
=
\begin{bmatrix}
\Phi_D(\mZ) & \Phi_C(\mZ) \\
\Phi_B(\mZ) & \Phi_A(\mZ)
\end{bmatrix}=\Phi(\mZ).
$$

Therefore, the first statement is true. 
The second statement reduces to the standard controller construction for LMI-based $\mathcal{H}_\infty$-synthesis~\cite{scherer1997multiobjective}. We complete the proof.
\end{proof}

\begin{remark}
The main difference between our proof and that in \cite[Proposition 3.1]{zheng2021analysis} is that we have used the bigger LMI conditions \cref{eq:LMI1} and \cref{eq:lmi3} for $\mathcal{H}_\infty$ control. Simpler LMI conditions was used in \cite{zheng2021analysis} since it focuses on stability. \hfill $\square$
\end{remark}

Based on \Cref{proposition:Phi_surjective},  any path-connected component of $\mathcal{G}_{\gamma}$ has a path-connected image under the surjective mapping $\Phi$. Consequently, the number of path-connected components of $\mathcal{K}_{\gamma}$ will be no more than the number of path-connected components of $\mathcal{G}_{\gamma}$.
The number of path-connected components of the set $\mathcal{G}_{\gamma}$ is given below.
\begin{proposition}\label{lemma:Gn_connected_components}
The set $\mathcal{G}_{\gamma}$ has two path-connected components, given by
\vspace{-1mm}
\begin{align*}
\mathcal{G}_{\gamma}^+
=\ &
\left\{(X,Y,\hat{\mathbf{A}},\hat{\mathbf{B}},\hat{\mathbf{C}},\hat{\mathbf{D}},\Pi,\Xi)\in\mathcal{G}_{\gamma} \mid \, \det \Pi>0\right\}, \\
\mathcal{G}_{\gamma}^-
=\ &
\left\{(X,Y,\hat{\mathbf{A}},\hat{\mathbf{B}},\hat{\mathbf{C}},\hat{\mathbf{D}},\Pi,\Xi)\in\mathcal{G}_{\gamma} \mid \, \det \Pi<0\right\}.
\end{align*}
\end{proposition}
\vspace{1mm}
\begin{proof}
First, $\mathcal{F}_{\gamma}$ is path-connected since it is convex.  The set of real invertible matrices $\mathrm{GL}_{n_x}\!=\!\{\Pi\in\mathbb{R}^{n_x\times n_x}\mid\,\det \Pi\!\neq\! 0\}$ has two path-connected components \cite{lee2013introduction}
\begin{align*}
\mathrm{GL}^+_{n_x}&=\{\Pi\in\mathbb{R}^{n_x\times n_x}\mid\,\det \Pi> 0\},\\
\mathrm{GL}^-_{n_x}&=\{\Pi\in\mathbb{R}^{n_x\times n_x}\mid\,\det \Pi< 0\}.
\end{align*}
Thus, the Cartesian product $\mathcal{F}_{\gamma}\times \mathrm{GL}_{n_x}$ has two path-connected components. We further observe that the mapping from $(X,Y,\hat{\mathbf{A}},\hat{\mathbf{B}},\hat{\mathbf{C}},\hat{\mathbf{D}},\Pi)$
to $(X,Y,\hat{\mathbf{A}},\hat{\mathbf{B}},\hat{\mathbf{C}},\hat{\mathbf{D}},\Pi,(I-YX)\Pi^{-1})$
is a continuous bijection from $\mathcal{F}_{\gamma}\times \mathrm{GL}_{n_x}$ to $\mathcal{G}_{\gamma}$. This immediately leads to the desired conclusion.
\end{proof}

We note that the proofs for \Cref{lemma:Gn_connected_components} and  \cite[Proposition 3.2]{zheng2021analysis} are similar. As a matter of fact, Proposition 3.2 in \cite{zheng2021analysis} can be viewed as a special case of \Cref{lemma:Gn_connected_components} with $\gamma\rightarrow+\infty$. 
Now \cref{Theo:disconnectivity} can be proved by combining \cref{proposition:Phi_surjective} with \cref{lemma:Gn_connected_components}.

\vspace{1mm}

\textbf{Proof of \cref{Theo:disconnectivity}:} 
 We define 
 $$\mathcal{K}_{\gamma}^{+}:=\Phi(\mathcal{G}_{\gamma}^+)\quad  \text{and} \quad  \mathcal{K}_{\gamma}^{-}:=\Phi(\mathcal{G}_{\gamma}^-).$$
Then, we have $\mathcal{K}_{\gamma} = \mathcal{K}_{\gamma}^{+} \cup \mathcal{K}_{\gamma}^{-}.$ If $\mathcal{K}_{\gamma}$ is not path-connected,
the two path-connected components of $\mathcal{K}_{\gamma}$ are exactly $\mathcal{K}_{\gamma}^{+}$ and $\mathcal{K}_{\gamma}^{-}$.  Based on \cref{proposition:Phi_surjective}, \cref{Theo:disconnectivity} holds. \hfill 
$\square$

\vspace{1mm}

In the next section, we further discuss some implications of \Cref{Theo:disconnectivity} on $\mathcal{H}_\infty$-constrained policy optimization.

\subsection{Implications for $\mathcal{H}_\infty$-constrained policy optimization} \label{subsection:implications}

To understand the implications of \cref{Theo:disconnectivity} for policy optimization, we need to formalize the relationship between $\mathcal{K}_{\gamma}^{+}$ and $\mathcal{K}_{\gamma}^{-}$. For this, we introduce the notion of similarity transformation that is widely used in control.  
For any $T\in\mathrm{GL}_{n_x}$, let $\mathscr{T}_T:\mathcal{C}_{\mathrm{stab}}\rightarrow\mathcal{C}_{\mathrm{stab}}$ denote the mapping given by
$$
\mathscr{T}_T(\mK)
:=\begin{bmatrix}
D_{\mK} & C_{\mK}T^{-1} \\
TB_{\mK} & TA_{\mK}T^{-1}
\end{bmatrix},
$$
which represents similarity transformations on $\mathcal{C}_{\mathrm{stab}}$. 

We have a result that is similar to \cite[Theorem 3.2]{zheng2021analysis}.

\begin{theorem} \label{Theo:Cn_homeomorphic}
{If ${\mathcal{K}}_{\gamma}$ has two path-connected components $\mathcal{K}_{\gamma}^{+}$ and $\mathcal{K}_{\gamma}^{-}$, then $\mathcal{K}_{\gamma}^{+}$ and $\mathcal{K}_{\gamma}^{-}$ are diffeomorphic under the mapping $\mathscr{T}_T$, for any  $T\in\mathrm{GL}_{n_x}$ with $\det T<0$. 
}
\end{theorem}

Furthermore, similar to \cite[Theorem 3.3]{zheng2021analysis},  we have sufficient conditions to certify the path-connectedness of $\mathcal{K}_{\gamma}$.

\begin{theorem} \label{Theo:connectivity_conditions}
Let $\gamma>\gamma^\star$. The following statements hold.
\begin{enumerate}
    \item 
$\mathcal{K}_{\gamma}$ is path-connected if it has one non-minimal dynamical controller.

\item Suppose the plant \cref{eq:Dynamic} is single-input or single-output, i.e., $m=1$ or $p=1$. The set ${\mathcal{K}}_{\gamma}$ is path-connected if and only if it has a non-minimal dynamical controller.
\end{enumerate}
\end{theorem}

The proofs of \Cref{Theo:Cn_homeomorphic,Theo:connectivity_conditions} are adapted from \cite{zheng2021analysis}, and we provide them in the appendix for completeness. 
\Cref{Theo:Cn_homeomorphic,Theo:connectivity_conditions} bring positive news on local policy search methods for $\mathcal{H}_\infty$-constrained optimization \cref{eq:policy-optimization}. If $\mathcal{K}_\gamma$ is path-connected, it makes sense to initialize the policy search from any point in the feasible set. If $\mathcal{K}_\gamma$ has two path-connected components, then the initial point may fall into either of the components. If the cost function $J(\mK)$ is invariant with respect to similarity transformations (e.g. the LQG cost), then both components include global minima. It becomes reasonable to initialize the policy search within either path-connected component. The following corollary is immediate. 

\begin{corollary}
Suppose the cost function $J(\mK)$ is invariant with respect to similarity transformations, then there exists a continuous path connecting any feasible point $\mK \in \mathcal{K}_\gamma$ to a global minimum of \cref{eq:policy-optimization} if it exists.  
\end{corollary}


\subsection{The case of strictly proper controllers}

We briefly discuss the case of strictly proper dynamical controllers with $D_{\mK}=0$, which is required in some classical control problems, including the continuous-time LQG problem \cite{zhou1996robust}. 
The topological properties of $\tilde{\mathcal{K}}_\gamma$ in \cref{eq:Hinfset} and $\mathcal{K}_\gamma$ in \cref{eq:Hinfset1} are identical.  
To see this, we let 
$$
    \begin{aligned}
    \tilde{\mathcal{F}}_\gamma & =\{(X,Y,\hat{\mathbf{A}},\hat{\mathbf{B}},\hat{\mathbf{C}},\hat{\mathbf{D}}) \in \mathcal{F}_\gamma \mid \hat{\mathbf{D}}=0\}, \\
    \tilde{\mathcal{G}}_\gamma & = \{(X,Y,\hat{\mathbf{A}},\hat{\mathbf{B}},\hat{\mathbf{C}},\hat{\mathbf{D}},\Pi,\Xi) \in \mathcal{G}_\gamma \mid \hat{\mathbf{D}}=0\}.
    \end{aligned}
$$
%
%
%
Minor modification of the proofs in \Cref{subsection:implications,subsection:proof-theorem-1} can show that $\tilde{\mathcal{F}}_\gamma$ is path-connected, and that $\tilde{\mathcal{G}}_\gamma$ has two path-connected components. The same mapping $\Phi$ in \cref{eq:change_of_variables_output} is a continuous and surjective mapping from $\tilde{\mathcal{G}}_\gamma$ to $\tilde{\mathcal{K}}_\gamma$. Therefore, we conclude that $\tilde{\mathcal{K}}_\gamma$ has at most two path-connected components and they are diffeomorphic under the similarity transformation with $\det(T)<0$. 

\section{Revisit sublevel sets in LQG and $\mathcal{H}_\infty$ control} \label{section:sublevel-set}

The results in \Cref{sec:connectivity} can be also interpreted as the connectivity of strict sublevel sets in optimal $\mathcal{H}_\infty$ control. 
Based on \eqref{eq:hinf-transfer-function}, $\mathbf{T}_{zw}$ can be viewed as a function of $\mK$, and the optimal $\mathcal{H}_\infty$ synthesis \cite{zhou1996robust} can be formulated as
\begin{equation} \label{eq:hinfcontrol}
    \begin{aligned}
    \min_{\mK} \quad  &\norm{\mathbf{T}_{zw}}_\infty \\
    \text{subject to} \quad & \mK \in \mathcal{C}_{\mathrm{stab}}.
    \end{aligned}
\end{equation}
Now, $\mathcal{K}_\gamma$ in \eqref{eq:Hinfset} is exactly the $\gamma$-level strict sublevel set of the optimal $\mathcal{H}_\infty$ control \eqref{eq:hinfcontrol}. Thus, \Cref{Theo:disconnectivity,Theo:Cn_homeomorphic,Theo:connectivity_conditions} characterize the strict sub-level sets of optimal $\mathcal{H}_\infty$ control. 

In addition to \eqref{eq:hinfcontrol}, the proof idea of using the change of variables \eqref{eq:change_of_variables_output} can be applied to other output feedback control problems to establish connectivity of their strict sublevel sets. For example, we can consider an $\mathcal{H}_2$ formulation of the LQG control~\cite{zheng2021analysis} as follows  
\begin{equation} \label{eq:h2control}
    \begin{aligned}
    \min_{\mK} \quad  &\norm{\mathbf{T}_{zw}}_{2}^2\\
    \text{subject to} \quad & \mK \in \mathcal{C}_{\mathrm{stab}} \cap \{\mK \mid D_{\mK}=0\},
    \end{aligned}
\end{equation}
where $\|\mathbf{T}_{zw}\|_{2}$ denotes the $\mathcal{H}_2$ norm of $\mathbf{T}_{zw}$. This problem \cref{eq:h2control} covers the LQG control as a special case when the~dynamics in \eqref{eq:Dynamic} are chosen appropriately  (see the appendix for details). Then, the same proof techniques in  \Cref{sec:connectivity} can establish the connectivity of the strict sublevel sets of \eqref{eq:h2control}:
\begin{equation} \label{eq:H2-levelset}
    \mathcal{L}_\gamma = \{\mK \in \mathcal{C}_{\mathrm{stab}} \mid D_{\mK}=0, \norm{\mathbf{T}_{zw}}_{2}^2 < \gamma \}.
\end{equation}
We have the following result (see the appendix for details). 
\begin{theorem} \label{theo:LQG-sublevel-set}
Under \cite[Assumption 1]{zheng2021analysis}, the strict sublevel set $\mathcal{L}_\gamma$ \eqref{eq:H2-levelset} has at most two path-connected components $\mathcal{L}_\gamma^{(1)}$ and $\mathcal{L}_\gamma^{(2)}$.  If so, $\mathcal{L}_\gamma^{(1)}$ and $\mathcal{L}_\gamma^{(2)}$ are diffeomorphic under the mapping $\mathscr{T}_T$, for any $T\in\mathrm{GL}_{n_x}$ with $\det T<0$. 
\end{theorem}

\begin{remark}
Path connectivity of sublevel sets may imply  some further landscape properties (e.g., critical points and uniqueness of minimizing sets)~\cite{ortega2000iterative,martin1982connected}. In particular, using a special~definition of minimizing sets in~\cite[Definition 5.1]{martin1982connected}, Theorem 5.4 in \cite{martin1982connected} guarantees that the $\mathcal{H}_\infty$ control \eqref{eq:hinfcontrol} and LQG control \eqref{eq:h2control}  have a unique global minimizing set in some weak sense (see the definition of LTMS in \cite{martin1982connected}). The notion of LTMS in~\cite[Definition 5.1]{martin1982connected} does not rule out the normal notion of local minima and saddle points; we refer the readers to \cite{martin1982connected} for detailed discussions. Indeed, it is shown in \cite{zheng2021analysis} that saddle points exist in \eqref{eq:h2control}. A rigorous definition of strict local minima for \eqref{eq:hinfcontrol} or \eqref{eq:h2control}  requires some extra work due to unboundedness of similarity transformations. \hfill $\square$
\end{remark}

\section{Numerical examples} \label{section:examples}
We present two simple examples to illustrate our main result (\Cref{Theo:disconnectivity}). We consider the open-loop unstable~system in \Cref{example1} with $A=B_1=B_2=C_1=C_2=D_{21}=D_{12}=1$, and  $D_{11}=0
$. To ease visualization, We consider strictly proper controllers. In the left plots of Figure \ref{fig:Ex1}, we visualize $\tilde{\mathcal{K}}_\gamma$ for $\gamma=50$ and $2$. We can see that $\tilde{\mathcal{K}}_\gamma$ has two path-connected components. As we decrease $\gamma$, the feasible region shrinks. In the right plots of Figure \ref{fig:Ex1}, we change the value of $A$ to $-1$, and visualize
$\tilde{\mathcal{K}}_\gamma$ for the resultant system. In this case, $\tilde{\mathcal{K}}_\gamma$ are path-connected for both values of~$\gamma$.
\begin{figure}[t!]
\centering
\includegraphics[width=\columnwidth]{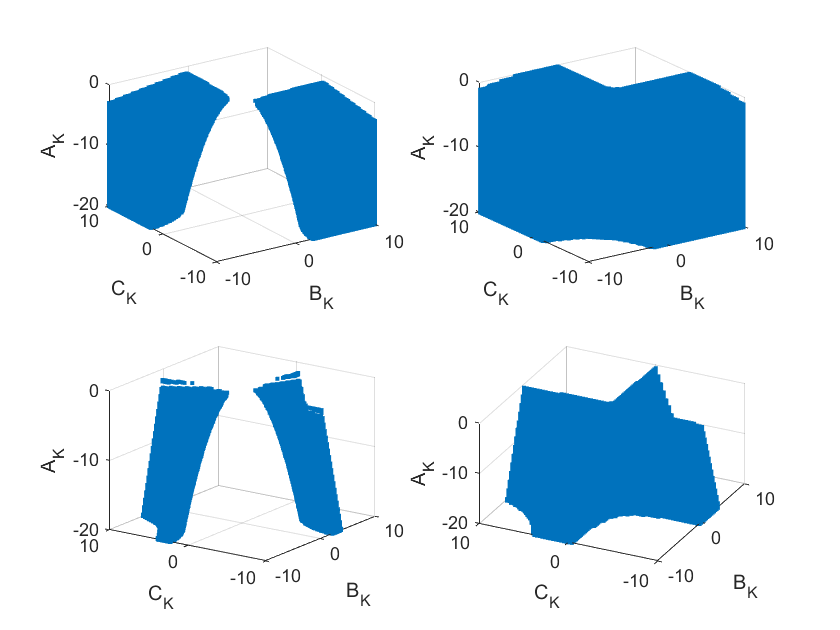}
\caption{The left two plots demonstrate the two path-connected components of $\mathcal{K}_\gamma$ for \Cref{example1}. The right plots show that $\mathcal{K}_\gamma$ becomes path-connected when $A$ is changed to $-1$ corresponding to an open-loop stable system.}
\label{fig:Ex1}
\end{figure}

\section{Conclusions} \label{section:conclusion}

We have proved that the set of $\mathcal{H}_\infty$-constrained full-order dynamical controllers has at most two path-connected components (cf. \Cref{Theo:disconnectivity}) and they are diffeomorphic under similarity transformations (cf. \Cref{Theo:Cn_homeomorphic}). We have also discussed various implications on direct policy search of robust dynamical controllers and on the strict sublevel sets of LQG and $\mathcal{H}_\infty$ control (cf. \Cref{theo:LQG-sublevel-set}). An important future direction is to develop provably convergent policy search methods for $\mathcal{H}_\infty$-constrained robust control problems.


\bibliographystyle{IEEEtran}
\bibliography{IEEEabrv,main}

\begin{thebibliography}{10}
\providecommand{\url}[1]{#1}
\csname url@samestyle\endcsname
\providecommand{\newblock}{\relax}
\providecommand{\bibinfo}[2]{#2}
\providecommand{\BIBentrySTDinterwordspacing}{\spaceskip=0pt\relax}
\providecommand{\BIBentryALTinterwordstretchfactor}{4}
\providecommand{\BIBentryALTinterwordspacing}{\spaceskip=\fontdimen2\font plus
\BIBentryALTinterwordstretchfactor\fontdimen3\font minus
  \fontdimen4\font\relax}
\providecommand{\BIBforeignlanguage}[2]{{%
\expandafter\ifx\csname l@#1\endcsname\relax
\typeout{** WARNING: IEEEtran.bst: No hyphenation pattern has been}%
\typeout{** loaded for the language `#1'. Using the pattern for}%
\typeout{** the default language instead.}%
\else
\language=\csname l@#1\endcsname
\fi
#2}}
\providecommand{\BIBdecl}{\relax}
\BIBdecl

\bibitem{pmlr-v80-fazel18a}
M.~Fazel, R.~Ge, S.~Kakade, and M.~Mesbahi, ``Global convergence of policy
  gradient methods for the linear quadratic regulator,'' in \emph{International
  Conference on Machine Learning}, vol.~80, 2018, pp. 1467--1476.

\bibitem{malik2019derivative}
D.~Malik, A.~Pananjady, K.~Bhatia, K.~Khamaru, P.~Bartlett, and M.~Wainwright,
  ``Derivative-free methods for policy optimization: {G}uarantees for linear
  quadratic systems,'' in \emph{International Conference on Artificial
  Intelligence and Statistics}, 2019, pp. 2916--2925.

\bibitem{mohammadi2021convergence}
H.~Mohammadi, A.~Zare, M.~Soltanolkotabi, and M.~R. Jovanovic, ``Convergence
  and sample complexity of gradient methods for the model-free linear quadratic
  regulator problem,'' \emph{IEEE Transactions on Automatic Control}, 2021.

\bibitem{furieri2020learning}
L.~Furieri, Y.~Zheng, and M.~Kamgarpour, ``Learning the globally optimal
  distributed {LQ} regulator,'' in \emph{Learning for Dynamics and Control},
  2020, pp. 287--297.

\bibitem{zhang2021policy}
K.~Zhang, B.~Hu, and T.~Basar, ``Policy optimization for $\mathcal{H}_2$ linear
  control with $\mathcal{H}_\infty$ robustness guarantee: Implicit
  regularization and global convergence,'' \emph{SIAM Journal on Control and
  Optimization}, vol.~59, no.~6, pp. 4081--4109, 2021.

\bibitem{zhang2020stability}
K.~Zhang, B.~Hu, and T.~Ba{\c{s}}ar, ``On the stability and convergence of
  robust adversarial reinforcement learning: A case study on linear quadratic
  systems,'' \emph{Advances in Neural Information Processing Systems}, vol.~33,
  2020.

\bibitem{gravell2020learning}
B.~Gravell, P.~M. Esfahani, and T.~Summers, ``Learning optimal controllers for
  linear systems with multiplicative noise via policy gradient,'' \emph{IEEE
  Transactions on Automatic Control}, vol.~66, no.~11, pp. 5283--5298, 2020.

\bibitem{zhang2021derivative}
K.~Zhang, X.~Zhang, B.~Hu, and T.~Ba{\c{s}}ar, ``Derivative-free policy
  optimization for linear risk-sensitive and robust control design: Implicit
  regularization and sample complexity,'' in \emph{Thirty-Fifth Conference on
  Neural Information Processing Systems}, 2021.

\bibitem{jansch2020convergence}
J.~P. Jansch-Porto, B.~Hu, and G.~E. Dullerud, ``Convergence guarantees of
  policy optimization methods for {M}arkovian jump linear systems,'' in
  \emph{American Control Conference}, 2020, pp. 2882--2887.

\bibitem{rathod2021global}
S.~Rathod, M.~Bhadu, and A.~De, ``Global convergence using policy gradient
  methods for model-free markovian jump linear quadratic control,'' \emph{arXiv
  preprint arXiv:2111.15228}, 2021.

\bibitem{jansch2020policyA}
J.~P. Jansch-Porto, B.~Hu, and G.~Dullerud, ``Policy optimization for markovian
  jump linear quadratic control: Gradient-based methods and global
  convergence,'' \emph{arXiv preprint arXiv:2011.11852}, 2020.

\bibitem{zhou1996robust}
K.~Zhou, J.~C. Doyle, and K.~Glover, \emph{Robust and optimal control}.\hskip
  1em plus 0.5em minus 0.4em\relax Prentice Hall, 1996.

\bibitem{sun2021learning}
Y.~Sun and M.~Fazel, ``Learning optimal controllers by policy gradient: Global
  optimality via convex parameterization,'' in \emph{2021 60th IEEE Conference
  on Decision and Control (CDC)}, 2021, pp. 4576--4581.

\bibitem{feng2019exponential}
H.~Feng and J.~Lavaei, ``On the exponential number of connected components for
  the feasible set of optimal decentralized control problems,'' in \emph{2019
  American Control Conference (ACC)}, 2019, pp. 1430--1437.

\bibitem{fatkhullin2020optimizing}
I.~Fatkhullin and B.~Polyak, ``Optimizing static linear feedback: Gradient
  method,'' \emph{SIAM Journal on Control and Optimization}, vol.~59, no.~5,
  pp. 3887--3911, 2021.

\bibitem{zheng2021analysis}
Y.~Zheng, Y.~Tang, and N.~Li, ``Analysis of the optimization landscape of
  linear quadratic gaussian ({LQG}) control,'' \emph{arXiv preprint
  arXiv:2102.04393}, 2021.

\bibitem{duan2022optimization}
J.~Duan, W.~Cao, Y.~Zheng, and L.~Zhao, ``On the optimization landscape of
  dynamical output feedback linear quadratic control,'' \emph{arXiv preprint
  arXiv:2201.09598}, 2022.

\bibitem{umenberger2022globally}
J.~Umenberger, M.~Simchowitz, J.~C. Perdomo, K.~Zhang, and R.~Tedrake,
  ``Globally convergent policy search over dynamic filters for output
  estimation,'' \emph{arXiv preprint arXiv:2202.11659}, 2022.

\bibitem{bu2019topological}
J.~Bu, A.~Mesbahi, and M.~Mesbahi, ``On topological and metrical properties of
  stabilizing feedback gains: the {MIMO} case,'' \emph{arXiv preprint
  arXiv:1904.02737}, 2019.

\bibitem{doyle1978guaranteed}
J.~C. Doyle, ``Guaranteed margins for {LQG} regulators,'' \emph{IEEE
  Transactions on Automatic Control}, vol.~23, no.~4, pp. 756--757, 1978.

\bibitem{Dullerud99}
G.~Dullerud and F.~Paganini, \emph{A Course in Robust Control Theory: A Convex
  Approach}.\hskip 1em plus 0.5em minus 0.4em\relax Springer, 1999.

\bibitem{whittle1990risk}
P.~Whittle, \emph{Risk-sensitive Optimal Control}.\hskip 1em plus 0.5em minus
  0.4em\relax Wiley Chichester, 1990.

\bibitem{gahinet1994linear}
P.~Gahinet and P.~Apkarian, ``A linear matrix inequality approach to
  $\mathcal{H}_\infty$ control,'' \emph{International journal of robust and
  nonlinear control}, vol.~4, no.~4, pp. 421--448, 1994.

\bibitem{scherer1997multiobjective}
C.~Scherer, P.~Gahinet, and M.~Chilali, ``Multiobjective output-feedback
  control via {LMI} optimization,'' \emph{IEEE Transactions on Automatic
  Control}, vol.~42, no.~7, pp. 896--911, 1997.

\bibitem{lee2013introduction}
J.~M. Lee, \emph{Introduction to Smooth Manifolds}, 2nd~ed.\hskip 1em plus
  0.5em minus 0.4em\relax Springer Science \& Business Media, 2013.

\bibitem{ortega2000iterative}
J.~M. Ortega and W.~C. Rheinboldt, \emph{Iterative solution of nonlinear
  equations in several variables}.\hskip 1em plus 0.5em minus 0.4em\relax SIAM,
  2000.

\bibitem{martin1982connected}
D.~H. Martin, ``Connected level sets, minimizing sets, and uniqueness in
  optimization,'' \emph{Journal of Optimization Theory and Applications},
  vol.~36, no.~1, pp. 71--91, 1982.

\end{thebibliography}

\balance
\newpage

\appendix

\subsection{Auxiliary proofs}

\textbf{Proof of \Cref{Theo:Cn_homeomorphic}:} 
The proof is similar to \cite[Theorem 3.2]{zheng2021analysis}. We provide provide a proof sketch below.  
It suffices to show that, 
for any $T\in\mathbb{R}^{n_x\times n_x}$ with $\det T<0$, the mapping $\mathscr{T}_T$ restricted on $\mathcal{K}_{\gamma}^{+}$ gives a diffeomorphism from $\mathcal{K}_{\gamma}^{+}$ to $\mathcal{K}_{\gamma}^{-}$.
We only need to show that 
$$\mathscr{T}_T(\mathcal{K}_{\gamma}^+)\subseteq \mathcal{K}_{\gamma}^-, \qquad \text{and} \qquad  \mathscr{T}_{T^{-1}}(\mathcal{K}_{\gamma}^-)\subseteq \mathcal{K}_{\gamma}^+,$$
when $\det T<0$. Consider any arbitrary $
\mK\in\mathcal{K}_{\gamma}^+$.
Then there exists $\mZ=(X,Y,\hat{\mathbf{A}},\hat{\mathbf{B}},\hat{\mathbf{C}},\hat{\mathbf{D}},\Pi,\Xi)\in\mathcal{G}_{\gamma}^+$ such that $\Phi(\mZ)=\mK$. We let
$$
\hat \Pi = T\Pi,\quad \hat \Xi = \Xi T^{-1},
\quad
\hat \mZ = (X,Y,\hat{\mathbf{A}},\hat{\mathbf{B}},\hat{\mathbf{C}},\hat{\mathbf{D}},\hat\Pi,\hat\Xi).
$$
Note that $\det\hat{\Pi}=\det T\cdot\det \Pi < 0$, leading to $\hat\mZ\in \mathcal{G}_{\gamma}^-$. We can further verify  $\mathscr{T}_T(\mK)=\Phi(\hat{\mZ})\in \Phi(\mathcal{G}_{\gamma}^-)=\mathcal{K}_{\gamma}^-$ and consequently $\mathscr{T}_{T}(\mathcal{K}_{\gamma}^+)\subseteq \mathcal{K}_{\gamma}^-$.  The proof of $\mathscr{T}_{T^{-1}}(\mathcal{K}_{\gamma}^-)\subseteq \mathcal{K}_{\gamma}^+$ is similar. This completes the proof.  \hfill $\square$

\textbf{Proof of \Cref{Theo:connectivity_conditions}:} 
If $\mathcal{K}_\gamma$ has a non-minimal dynamical controller, then there exists a reduced-order stabilizing controller with an internal state dimension $(n_x-1)$, satisfying the $\mathcal{H}_\infty$ constraint.
Denote the state/input/output matrices of this controller as $(\tilde{A}_{\mK}, \tilde{B}_{\mK}, \tilde{C}_{\mK},\tilde{D}_{\mK})$. Then, this controller can be augmented to be a full-order controller in $\mathcal{K}_{\gamma}$ as
$$
\mK = 
\begin{bmatrix}
\tilde{D}_{\mK} & \tilde{C}{_\mK} & 0 \\
\tilde{B}{_\mK} & \tilde{A}{_\mK} & 0 \\
0 & 0 & -1
\end{bmatrix}
$$
Define a similarity transformation matrix
$$
T=\begin{bmatrix}
I_{n_x-1} & 0 \\ 0 & -1
\end{bmatrix}.
$$
By the proof of \cref{Theo:Cn_homeomorphic}, we can see that $\mK\in\mathcal{K}_{\gamma}^\pm$ implies $\mathscr{T}_T(\mK)\in\mathcal{K}_{\gamma}^\mp$. On the other hand, we can directly check that $\mathscr{T}_T(\mK)=\mK$. Therefore, we have 
$\mK\in \mathcal{K}_{\gamma}^+\cap \mathcal{K}_{\gamma}^-,$  
indicating that $\mathcal{K}_{\gamma}^+\cap \mathcal{K}_{\gamma}^-$ is nonempty. Consequently, $\mathcal{K}_{\gamma}$ is path-connected.

The proof for the second statement is identical to the proof of \cite[Theorem 3.3]{zheng2021analysis}, and hence is omitted here.
\hfill $\square$

\subsection{Connectivity of strict sublevel sets for LQG}
In this section, we briefly discuss the path-connectivity of the strict sublevel sets for LQG and present the proof for \Cref{theo:LQG-sublevel-set}. 

Consider the LTI system \cref{eq:Dynamic}. We can exactly recover the LQG setup in \cite{zheng2021analysis} by choosing state, input, and output matrices as
\begin{align*}
B_1&=\bmat{W^{\frac{1}{2}} & 0}, C_1=\bmat{Q^{\frac{1}{2}}\\ 0}, D_{11}=0, D_{12}=\bmat{0 \\ R^{\frac{1}{2}}},\\
D_{21}&=\bmat{0 & V^{\frac{1}{2}}}, D_{\mK}=0.
\end{align*}
where $W\succeq 0$, $Q\succeq 0$, $R\succ 0$, and $V\succ 0$. Then the LQG problem in \cite{zheng2021analysis} can be equivalently formulated as \cref{eq:h2control}. 
Since strictly proper controllers are used, we always have $D_{\mathrm{cl}}=0$.

We now present the proof for \Cref{theo:LQG-sublevel-set}.

\textbf{Proof of \Cref{theo:LQG-sublevel-set}:} 
It is well-known that
we have $\mK\in \mathcal{L}_\gamma$ if and only if there exist $P$ and $\Gamma$ such that,
\begin{align}
\label{eq:LMI1_H2}
 \begin{split}
 & \bmat{ A_{\mathrm{cl}}^\tr P+P A_{\mathrm{cl}} & P B_{\mathrm{cl}} \\  B_{\mathrm{cl}}^\tr P & -I}  \prec 0, \,\, \\
 & \bmat{P & C_{\mathrm{cl}}^\tr\\ C_{\mathrm{cl}} & \Gamma}\succ 0,\;\; \Tr(\Gamma)<\gamma.
 \end{split}
\end{align}
The above condition is not convex in $\mK$ and $P$.  
However, we can use the same change of variables as \eqref{eq:change_of_variables_output} in the main text. A controller $\mK\in \mathcal{L}_\gamma$ can be constructed if  $\exists (X,Y,\hat{\mathbf{A}},\hat{\mathbf{B}},\hat{\mathbf{C}},\Gamma)$ such that the following LMI holds\footnote{Since strictly proper controllers are considered, we always have $\hat{\mathbf{D}}=0$. We thus get rid of this matrix $\hat{\mathbf{D}}$ and only $(\hat{\mathbf{A}},\hat{\mathbf{B}},\hat{\mathbf{C}})$ show up in the LMI.},

\begin{align}\label{eq:H2LMI}
\begin{split}
  &\bmat{\mM_{11}^{(\mathrm{LQG})} & \hat{\mathbf{A}}^\tr+A & B_1\\
    \hat{\mathbf{A}}+A^\tr & \mM_{22}^{(\mathrm{LQG})} & YB_1+\hat{\mathbf{B}}D_{21}\\
    B_1 & (YB_1+\hat{\mathbf{B}}D_{21})^\tr & -I} \! \prec 0,\\
    &\begin{bmatrix}
X & I & (C_1 X+D_{12} \hat{\mathbf{C}})^\tr \\ I & Y & C_1^\tr \\ C_1 X+D_{12} \hat{\mathbf{C}} & C_1 & \Gamma
\end{bmatrix} \! \succ 0,\\
&\Tr(\Gamma)<\gamma,
    \end{split}
\end{align}
where the blocks $\mM_{11}^{(\mathrm{LQG})}$ and $\mM_{22}^{(\mathrm{LQG})}$ are defined as
\begin{align*}
    \begin{split}
        \mM_{11}^{(\mathrm{LQG})}&=AX+XA^\tr+B_2 \hat{\mathbf{C}}+ (B_2 \hat{\mathbf{C}})^\tr,\\
        \mM_{22}^{(\mathrm{LQG})}&=A^\tr Y+YA+\hat{\mathbf{B}} C_2+(\hat{\mathbf{B}}C_2)^\tr.
    \end{split}
\end{align*}
Similarly, we  can define the following set:
\begin{align*} 
\mathcal{F}_\gamma^{(\mathrm{LQG})}&:=
\bigg\{
(X,Y,\hat{\mathbf{A}},\hat{\mathbf{B}},\hat{\mathbf{C}},\Gamma)
\mid  \text{\cref{eq:H2LMI} is satisfied}
\bigg\},
\end{align*}
which is obviously convex and path-connected.
Next, we define the set $\mathcal{G}_\gamma^{(\mathrm{LQG})}$ as
\begin{align*}
\begin{split}
\mathcal{G}_\gamma^{(\mathrm{LQG})}:=
\bigg\{&(X,Y,\hat{\mathbf{A}},\hat{\mathbf{B}},\hat{\mathbf{C}},\Gamma,\Pi,\Xi)
\mid \Pi,\Xi\in\mathbb{R}^{n_x \times n_x},\, \\&(X,Y,\hat{\mathbf{A}},\hat{\mathbf{B}},\hat{\mathbf{C}},\Gamma)\in\mathcal{F}_{\gamma}^{(\mathrm{LQG})},\, \Xi\Pi =I-YX\bigg\}.
\end{split}
\end{align*}
Similar to \Cref{lemma:Gn_connected_components}, we claim that $\mathcal{G}_\gamma^{(\mathrm{LQG})}$ has two path-connected components. 

Now we need to prove the mapping in \cref{eq:change_of_variables_output} with $\hat{\mathbf{D}}=0$ is a continuous and surjective mapping from $\mathcal{G}_\gamma^{(\mathrm{LQG})}$ to $\mathcal{L}_\gamma$. The continuity part is trivial. To show the surjectiveness, we need to prove the following statements. 
\begin{enumerate}
    \item For any arbitrary strictly proper controller $ \mK \in \mathcal{L}_\gamma$, there exists $\mZ=(X,Y,\hat{\mathbf{A}},\hat{\mathbf{B}},\hat{\mathbf{C}},\Gamma,\Pi,\Xi)\in\mathcal{G}_{\gamma}^{(\mathrm{LQG})}$ such that 
    $
        \Phi(\mZ) = \mK
    $.
    \item For all $ \mZ=(X,Y,\hat{\mathbf{A}},\hat{\mathbf{B}},\hat{\mathbf{C}},\Gamma,\Pi,\Xi)\in\mathcal{G}_{\gamma}^{(\mathrm{LQG})}$, we have $\Phi(\mZ)\in \mathcal{L}_\gamma$.
\end{enumerate}
The second statement reduces to the standard controller reconstruction for LMI-based $\mathcal{H}_2$ synthesis, and hence is known to be true. To prove the first statement, let $\mK=\bmat{0 & C_{\mK} \\ B_{\mK} & A_{\mK}}\in \mathcal{L}_\gamma$ be arbitrary. Then there exists $P\succ 0$ such that \cref{eq:LMI1_H2} is feasible. We partition $P$ as \cref{eq:Ydef}, and define $(X,\Pi, T)$ via \cref{eq:Xdef}. Then we still have $YX+\Xi\Pi=I$. Based on \cref{eq:LMI1_H2}, we have
\begin{align*}
&\bmat{T^\tr & 0 \\ 0 & I }  \bmat{A_{\mathrm{cl}}^\tr P+P A_{\mathrm{cl}} & P B_{\mathrm{cl}} \\ B_{\mathrm{cl}}^\tr P & - I}  \bmat{T & 0 \\ 0 & I}  \prec 0,\\
&\bmat{T^\tr & 0 \\ 0 & I}\bmat{P & C_{\mathrm{cl}}^\tr\\ C_{\mathrm{cl}} & \Gamma}\bmat{T & 0 \\ 0 & I}\succ 0,\,\Tr(\Gamma)<\gamma,
\end{align*}
which exactly reduces to \cref{eq:H2LMI} if we choose $(\hat{\mathbf{A}},\hat{\mathbf{B}},\hat{\mathbf{C}})$ as defined in \cref{eq:change_of_var1}
with $D_{\mK}=0$. Thus, we have $\mZ=(X,Y,\hat{\mathbf{A}},\hat{\mathbf{B}},\hat{\mathbf{C}},\Gamma,\Pi,\Xi)\in\mathcal{L}_{\gamma}$ by the definition of $\mathcal{L}_{\gamma}$.
Now the first statement holds as desired.

Therefore, the number of the path-connected components of $\mathcal{L}_\gamma$ cannot be larger than the number of the path-connected components of $\mathcal{G}_\gamma^{(\mathrm{LQG})}$. Finally, we can slightly modify the proof of \Cref{Theo:Cn_homeomorphic} to show that the two path-connected components are diffeomorphic under similarity transformations. This completes the proof.
\hfill $\square$



\end{document}